\documentclass{amsart}
\usepackage{color}
\usepackage[utf8]{inputenc}
\usepackage{amscd}
\usepackage{pstricks}
\usepackage{amsmath}
\usepackage[english]{babel}
\usepackage{amsfonts}
\usepackage{graphicx}
\usepackage{vmargin}
\usepackage{theoremref}
\usepackage[hidelinks]{hyperref}
\usepackage{amssymb}
\usepackage{tikz}
\usepackage{relsize}
\usepackage[toc,page]{appendix}
\usepackage{commath}
\usepackage{mathrsfs}
\usetikzlibrary{matrix}
\usepackage{cancel}
\usepackage{lineno}
\usepackage{comment}

\newcommand{\rank}{{\mathrm{rank}}}

\newcommand{\Id}{\mathrm{Id}}
\newcommand{\inv}{\mathrm{inv}}
\newcommand{\ninv}{\mathfrak{inv}}
\newcommand{\e}{\mathfrak{e}}

\newcommand{\stab}{\mathrm{stab}}
\newcommand{\im}{\mathrm{im}}
\newcommand{\rimes}{\!\times\! }
\renewcommand{\P}{\mathrm{P}}

\newtheorem{theorem}{Theorem}[section]
\newtheorem{lemma}[theorem]{Lemma}
\newtheorem{conjecture}[theorem]{Conjecture}
\newtheorem{definition}[theorem]{Definition}
\newtheorem{proposition}[theorem]{Proposition}
\newtheorem{corollary}[theorem]{Corollary}

\newtheorem{example}[theorem]{Example}
\newtheorem{notation}[theorem]{Notation}
\newtheorem{remark}[theorem]{Remark}

\usepackage{chngcntr}
\counterwithin{figure}{section}

\numberwithin{equation}{section}

\title[]{Groupoid morphisms as an algebraic structure for nonautonomous dynamics}
\author[]{N\'estor Jara}

\address{Universidad de Chile, Departamento de Matem\'aticas. Casilla 653, Santiago, Chile}

\email{nestor.jara@ug.uchile.cl}
\subjclass[2020]{18B40, 22A22, 37B55, 37C60.}
\keywords{Nonautonomus dynamics, Dynamical systems, Groupoids}
\thanks{This research has been partially supported by ANID, Beca de Doctorado Nacional 21220105.}
\date{\today}

\begin{document}

\maketitle

\begin{abstract}
We present groupoid morphisms as an algebraic structure for nonautonomous dynamics, as well as a generalization of group morphisms, which describe classic dynamical systems. We introduce the structure of cotranslations, as a specific kind of groupoid morphism, and establish a correspondence between cotranslations and skew-products. We give applications of cotranslations to nonautonomous equations, both in differences and differential. We obtain results about the differentiability of cotranslations, as well as dimension invariance and diagonalization (through a generalized notion of kinematic similarity) for a partial version of them, admitting noninvertible transformations. 
\end{abstract}

\section{Introduction}

Consider a category $\mathscr{C}$ (for instance, topological spaces, vector spaces, Banach spaces, among others) and an element $X$ on said category. We also consider
\begin{itemize}
    \item $\mathbb{B}_X$ the collection of all the morphisms of the category $\mathscr{C}$ of $X$ on itself,
    \item $\mathbb{A}_X$ the collection of all invertible elements of $\mathbb{B}_X$ (whose inverse is also in $\mathbb{B}_X$). We know that $\mathbb{A}_X$ is a group.
\end{itemize}

Depending on the category $\mathscr{C}$, both sets $\mathbb{B}_X$ and $\mathbb{A}_X$ may have more structure, but for now we keep it general.

\smallskip
A dynamical system, independently of the category $\mathscr{C}$, considers a group $G$ (with maybe some extra structure) and a group morphism $\gamma:G\to \mathbb{A}_X$. Particularly, in the topological case, where $X$ is a topological space and $G$ a topological group (for simplicity, let us consider both $X$ and $G$ to be locally compact and Hausdorff), we set on $\mathbb{A}_X$ the compact-open topology \cite[Definition I, p. 301]{Bourbaki}, and the group morphism $\gamma$ is required to be continuous.

\smallskip
Equivalently, we may define that a left (topological) dynamical system is a triple $(X,G,\alpha)$, where $X$ and $G$ satisfy the same conditions as before and $\alpha$ is a continuous {\it left action} of $G$ on $X$, {\it i.e.} a map $\alpha:G\rimes X\to X$ verifying that each $\alpha(g,\cdot)$ is a homeomorphism of $X$ on itself and
$$\alpha(g,\alpha(h,x))=\alpha(gh,x), \qquad \forall\,g,h\in G,\, x\in X.$$

It is easy to see that by setting $\hat{\alpha}:G\to \mathbb{A}_X$ by $\left[\hat{\alpha}(g)\right](x)=\alpha(g,x)$ we obtain a continuous group morphism.

\smallskip
On the other hand, we can define {\it right actions} as maps $\beta:X\times G\to X$ verifying that each $\beta(\cdot,g)$ is a homeomorphism of $X$ on itself and
$$\beta(\beta(x,g),h)=\beta(x,gh), \qquad \forall\,g,h\in G,\, x\in X.$$

Once more, by setting $\hat{\beta}:G\to \mathbb{A}_X$ by $\left[\hat{\beta}(g)\right](x)=\beta(x,g^{-1})$ we obtain a continuous group morphism (although, if $G$ is Abelian, we can just use $g$ instead of $g^{-1}$). 

\smallskip
Thus, both left and right actions describe dynamical systems and the use of one over the other is just a matter of convenience in notation. This correspondence between left or right actions and their algebraic counterpart on group morphisms is a well known fact at the basis of classic dynamics.

\smallskip
A very important, on its own right, example of dynamical system is given by the flow of solutions of an autonomous differential equation (satisfying standard existence and 
uniqueness conditions). To rephrase it, an autonomous differential equation defines a continuous action of $\mathbb{R}$ on (for instance) $\mathbb{R}^d$. We usually call such system an autonomous dynamic.

%The flow of solutions of an autonomous differential equation (satisfying standard existence and uniqueness conditions), for instance on $\mathbb{R}^d$, defines a continuous action of $\mathbb{R}$ on $\mathbb{R}^d$ (since $\mathbb{R}$ is Abelian, it does not matter if it is left or right), thus defining a dynamical system. However, this is not the case for nonautonomous differential equations. Actually, for nonautonomous equations, it is not so widely discussed what is the notion of {\it action}, even less what could be its algebraic counterpart.

\smallskip
However, as soon as we consider a nonautonomous differential equations, dynamical systems (hence group morphisms and actions) no longer describe properly the dynamics given by the flow of solutions. In other words, nonautonomous dynamics \cite{Kloeden} do not have a widely discussed analogous notion of {\it action}, even less what could be its algebraic counterpart.

\smallskip
The main goal of this paper is to present an algebraic generalization of dynamical systems, which can be applied to describe the flows given by nonautonomous equations (both differential and in differences), as well as generalize the notion of nonautonomous dynamics to general groups, and not just $\mathbb{R}$ or $\mathbb{Z}$.

\smallskip
The paper is organized as follows. In the second section we study
the structure of {\it skew-product}, objects which, to the best of our knowledge, emerge for the first time on 1950 by H. Anzai \cite{Anzai}, whom uses them to describe a certain ergodic dynamic, and then later on 1965 were connected to differential equations thanks to the work of R. K. Miller \cite{Miller}. This concept refers to a generalization of dynamical systems given by left actions, but instead of considering one action, it uses a family of action-like functions with a certain compatibility relation (we give more details later). The skew-products we present here are defined for general groups, not just the usual $\mathbb{R}$ or $\mathbb{Z}$.

\smallskip
On the third section we present groupoids and groupoid morphisms. We define a specific type of these groupoid morphisms called cotranslations, which will give an algebraic structure of the dynamics that can be represented by skew-products. We also state some results regarding the relation of these groupoid morphisms to  discrete nonautonomous dynamics. The virtue of cotranslations over skew-products is that more analytic and algebraic properties are easily deduced from this structure.

\smallskip
On the fourth section we study differentiable groupoid morphisms, give some basic properties and give an application to the problem of existence of solutions to linear nonautonomous differential equations on Banach spaces.

\smallskip
On the final section we study a partial notion of the groupoid morphisms we presented earlier, mainly on the Euclidean space, and give some of their algebraic and topological properties.

\section{Skew-product dynamical systems}

On this section we present skew-products dynamical systems. Although we present them in the topological case, a similar structure can be defined and studied for objects on different categories. As described by R. J Sacker \cite{Sacker2} on a symposium on 1976, a  (topological) skew-product is constituted from a locally compact Hausdorff topological space $X$ and a locally compact Hausdorff topological group $G$, both will be fixed on this section unless stated otherwise. Thus, in this context, $\mathbb{A}_X$ denotes the topological group of all homeomorphisms of $X$ on itself, given the compact-open topology. The main characteristic of this generalization of dynamical systems is that instead of a unique group action, we consider a family of action-like maps. It is worth noting that in \cite{Elaydi,Sacker,Sacker2} the groups are $\mathbb{R}$ or $\mathbb{Z}$, while here we present the construction for any locally compact Hausdorff topological group.

\smallskip
For $A,B$ locally compact Hausdorff spaces, we set $\mathcal{C}\big(A;B\big)$ the space of continuous functions from $A$ to $B$, given the compact-open topology. Let us consider the evaluation maps:
\begin{itemize}
\item $\mathfrak{e}:\mathcal{C}\big(G\!\times\!X; X\big)\times G\times X\to X$, $(\psi,g,x)\mapsto \psi(g,x)$,
    \item $\mathfrak{e}_G:\mathcal{C}\big(G; X\big)\times G\to X$, $(\varphi,g)\mapsto \varphi(g)$,
    \item $\mathfrak{e}_X:\mathcal{C}\big(X; X\big)\times G\to X$, $(\varphi,x)\mapsto \varphi(x)$,
\end{itemize}
which are all continuous \cite[Corollary I, p. 303]{Bourbaki}, as well as the partial evaluation maps:
\begin{itemize}
    \item $\widetilde{\mathfrak{e}}_G:\mathcal{C}\big(G\!\times\!X; X\big)\times G\to \mathcal{C}\big(X; X\big)$, $(\psi,g)\mapsto \psi(g,\cdot)$,
    \item $\widetilde{\mathfrak{e}}_X:\mathcal{C}\big(G\!\times\!X; X\big)\times X\to \mathcal{C}\big(G; X\big)$, $(\psi,x)\mapsto \psi(\cdot,x),$
\end{itemize}
which are also continuous \cite[Corollary II, p. 303]{Bourbaki}. For a subspace $Y\subset \mathcal{C}\big(G\!\times\!X; X\big)$, we define 
$$Y_X:=\widetilde{e}_X(Y\rimes X)\subset \mathcal{C}\big(G; X\big),\quad \text{and}\quad Y_G:=\widetilde{e}_G(Y\rimes G)\subset \mathcal{C}\big(X; X\big),$$ 
both as topological subspaces. It is easy to see that
$$\e(\psi,g,x)=\e_G\left(\widetilde{\e}_X(\psi,x),g\right)=\e_X\left(\widetilde{\e}_G(\psi,g),x\right),\quad\forall\,\psi\in Y,g\in G,x\in X.$$

\begin{comment}

It is easy to see that $Y$ is in fact homeomorphic to $\mathcal{C}\big(G;Y_G\big)$ and $\mathcal{C}\big(X;Y_X\big)$ \cite[Corollary II, p. 303]{Bourbaki}. The the diagram on Figure \ref{136} shows the restriction of the evaluation maps, which commutes tautologically.
\begin{figure}[h]
\centering
\begin{tikzpicture}
    \matrix (m) [matrix of math nodes,row sep=3em,column sep=1em,minimum width=2em]
    {
    &X\times Y\times G&\\
    X\times Y_G&&Y_X\times G\\
     &X&\\};
     \path[-stealth]
    (m-1-2) edge node [left] {$\Id_X\!\times\! \widetilde{e}_G\qquad$} (m-2-1)
    (m-2-1) edge node [left] {$\mathfrak{e}_X\quad$} (m-3-2)
    (m-2-3) edge node [right] {$\quad\mathfrak{e}_G$} (m-3-2)
    (m-1-2) edge node [right] {$\qquad\widetilde{e}_X\!\times\! \Id_G$} (m-2-3)
    (m-1-2) edge node [left] {$\mathfrak{e}$} (m-3-2);
    \end{tikzpicture}
    \caption{Restriction of evaluation maps.}\label{136}
\end{figure} 
\end{comment}
\begin{definition}
    We say that a space of functions $Y\subset \mathcal{C}\big(G\!\times\!X; X\big)$ is \textbf{admissible}, if:
    \begin{itemize}
        \item [i)] $Y_G\subset\mathbb{A}_X$
        \item [ii)] $\psi(e,x)=x$ for every $\psi\in Y$ and $x\in X$, where $e$ is the unit of $G$.
    \end{itemize}
    In \cite{Sacker}, such a space $Y$ is called a \textbf{Hull}.
\end{definition}

On the other hand, if we consider the right action $\theta$ of $G$ on itself by translations, {\it i.e.}
$$\theta:G\times G\to G,\quad (g,h)\mapsto gh,$$
it lifts to a left action $\Theta$ of $G$ on spaces of functions with domain $G$. In particular, we have $\Theta:G\!\times\!\mathcal{C}\big(G; X\big)\to \mathcal{C}\big(G; X\big)$ given by
$$\left[\Theta(h,\varphi)\right](g)=\varphi\left(\theta(g,h)\right)=\varphi(gh),\quad \forall\,\varphi\in \mathcal{C}\big(G; X\big),\,g,h\in G,$$
and an easy topology exercise shows that $\Theta$ is continuous. 
For an admissible collection $Y$, the set $Y_X$ is contained on $\mathcal{C}\big(G;X\big)$, thus we can consider its saturation $\widetilde{Y}_X:=\Theta\big(G\!\times\!Y_X\big)$, which is invariant under the action $\Theta$, hence $\left(\widetilde{Y}_X,G,\Theta\right)$ is a (left) dynamical system.

\smallskip
With this system, we may write the map $\widetilde{\Theta}:X\rimes Y\rimes G\rimes G\to X$ by 
\begin{equation}\label{192}
    \widetilde{\Theta}(x,\psi,g,h)=\left[\Theta\left(h,\widetilde{\e}_X(\psi,x)\right)\right](g)=\psi(gh,x).
\end{equation}

Now, consider a (topological) dynamical system $(Y,G,\sigma)$, given by a continuous left action $\sigma:G\times Y\to Y$. The {\it skew-flow} associated to it is the map $\pi:X\times Y\times G\to X\times Y$ given by $\pi=\mathfrak{e}\times\sigma$, that is
$$\pi(x,\psi,h)=\left(\mathfrak{e}(\psi,h,x),\sigma(h,\psi)\right)=\left(\psi(h,x),\sigma(h,\psi)\right),\quad\forall\,\psi\in Y,\,h\in G,\,x\in X,$$
which is the way in which skew-products are depicted on \cite{Sacker}. With this system, we may write the map $\widetilde{\pi}:X\rimes Y\rimes G\rimes G\to X$ by
\begin{equation}\label{193}
\widetilde{\pi}(x,\psi,g,h)=\left[\sigma\big(h,\psi\big)\right]\left(g,\psi\big(h,x\big)\right),\quad\forall\,\psi\in Y,\,g,h\in G,\,x\in X.
\end{equation}

Now, we have all the components of a skew-product, all that is left is to state the compatibility of the systems given by $\Theta$ and $\sigma$. This means that the maps $\widetilde{\Theta}$ and $\widetilde{\pi}$ on (\ref{192}) and (\ref{193}) respectively, must coincide, {\it i.e.}
$$\left[\sigma(h,\psi)\right]\left(g,\psi(h,x)\right)=\psi(gh,x) ,\quad\forall\,\psi\in Y,\,g,h\in G,\,x\in X.$$

In this last statement, it is clear that all the information regarding the skew-product is contained on the properties of $\sigma$. We formalize this discussion on the following definition:
\begin{definition}\label{300}
    A \textbf{skew-product dynamical system} is a quadruple $(X,G,Y,\sigma)$, where:
    \begin{itemize}
    \item [i)] $Y\subset \mathcal{C}\big(G\!\times\!X;X\big)$ is admissible,
        \item [ii)] $(Y,G,\sigma)$ is a dynamical system, where $\sigma:G\!\times\! Y\to Y$ is a continuous left action,
        \item [iii)] $\left[\sigma(h,\psi)\right]\left(g,\psi(h,x)\right)=\psi(gh,x)$ for every $\psi\in Y,\,g,h\in G$ and $\,x\in X$.
    \end{itemize}
\end{definition}

\begin{remark}%\label{137}
   {\rm  It is clear that if $(X,G,\alpha)$ is a topological dynamical system where $\alpha$ is a continuous left action, then $(X,G,\{\alpha\},\mathfrak{q})$ is a skew-product, where $\mathfrak{q}:G\rimes\{\alpha\}\to \{\alpha\}$ is the only possible action, {\it i.e.} the trivial action. In other words, a dynamical system is a skew-product where the hull contains only one function.}
\end{remark}

Let us illustrate this construction on an example.

\begin{example}%\label{130}
{\rm
Consider $\mathbb{R}^d$ as a topological space, hence $\mathbb{A}_X$ is the group of all of its homeomorphisms. Set $G=\mathbb{Z}$. Consider the following nonautonomous difference equation 
    \begin{equation}\label{121}
        x(n+1)=F\left(n,x(n)\right),
    \end{equation} 
    where for every $n\in \mathbb{Z}$ we have $F(n,\cdot)\in \mathbb{A}_X$. Let $n\mapsto x(n,m,\xi)$ be its unique solution such that $x(m,m,\xi)=\xi$. Define $Y=\left\{\psi_{m}\,:\,m\in \mathbb{Z}\right\}$, the collection of all solutions of (\ref{121}) parameterized by its {\it temporal} initial condition, {\it i.e.}
    $$\psi_{m}(n,\xi)=x\big(n+m,m,\xi\big).$$ 
    
    Set now $\sigma:Y\!\times\! \mathbb{Z}\to Y$ the action of {\it left translations on the temporal initial condition}, that is 
    $$\sigma(n,\psi_{m})=\psi_{n+m},\quad\forall\,n,m\in \mathbb{Z},$$ 
    or, evaluating
    \begin{align*}
    \left[\sigma\left(\psi_{m},n\right)\right](p,\xi)=\psi_{n+m}(p,\xi)=x\big(p+n+m,n+m,\xi\big),\quad\forall\,p,n,m\in \mathbb{Z},\,\xi\in \mathbb{R}^d,
    \end{align*}
    and its associated skew-flow $\pi:\mathbb{R}^d\!\times\!Y\!\times\!\mathbb{Z}\to \mathbb{R}^d\!\times\!Y$ given by $\pi\left(\xi,\psi_{m},n\right)=\left(\psi_{m}(n,\xi),\psi_{n+m}\right)$. Evaluating, we have
\begin{align*}
        \widetilde{\pi}(\xi,\psi_m,p,n)&=\psi_{n+m}\left(p,\psi_{m}\big(n,\xi\big)\right)\\
        &=x\left(p+n+m,n+m,x\big(n+m,m,\xi\big)\right),\quad\forall\,p,n,m\in \mathbb{Z},\,\xi\in \mathbb{R}^d.
    \end{align*}

    \begin{comment}

    \begin{align*}
        \mathfrak{e}\left[\pi\big(\xi,\psi_{m},n\big),p\right]&=\mathfrak{e}\left(\psi_{m}\big(n,\xi\big),\psi_{n+m},p\right)\\
        &=\psi_{n+m}\left(p,\psi_{m}\big(n,\xi\big)\right)\\
        &=x\left(p+n+m,n+m,x\big(n+m,m,\xi\big)\right),\quad\forall\,p,n,m\in \mathbb{Z},\,\xi\in \mathbb{R}^d.
    \end{align*}
\end{comment}
On the other hand, note that
\begin{align*}
    \widetilde{\Theta}(\xi,\psi_m,p,n)=\psi_{m}(p+n,\xi)=x\big(p+n+m,m,\xi\big),\quad\forall\,p,n,m\in \mathbb{Z},\,\xi\in \mathbb{R}^d.
\end{align*}

Now, it is well know that, by 
uniqueness of solutions, we have
$$x\left(p+n+m,n+m,x\big(n+m,m,\xi\big)\right)=x\big(p+n+m,m,\xi\big), \quad\forall\,p,n,m\in \mathbb{Z},\,\xi\in \mathbb{R}^d,$$
thus $\widetilde{\Theta}$ and $\widetilde{\pi}$ coincide, which implies that $\left(\mathbb{R}^d,\mathbb{Z},\{\psi_m:m\in \mathbb{Z}\},\sigma\right)$ is indeed a skew-product dynamical system.   }    
\begin{comment}
    $$\mathfrak{e}\left[\pi\big(\xi,\psi_{m},n\big),p\right]=\hat{\Theta}\left[\widetilde{\mathfrak{e}}_X\big(\xi,\psi_m\big),n,p\right],\quad\forall\,p,n,m\in \mathbb{Z},\,\xi\in \mathbb{R}^d,$$
which implies that $\left(\mathbb{R}^d,\mathbb{Z},\{\psi_m:m\in \mathbb{Z}\},\sigma\right)$ is indeed a skew-product dynamical system.   
\end{comment}
\end{example}

On a presentation given on 2004, S. Elaydi and R. J. Sacker \cite{Elaydi} give further applications of skew-products on the theory of nonautonomous difference equations, as the search for asymptotically stable solutions for Beverton-Holt equations \cite{Cushing}. On the other hand, in \cite{Sacker} the authors use this structure to develop the exponential dichotomy spectrum for nonautonomous linear differential equations when the admissible space of functions is compact.

\section{Groupoid morphisms and cotranslations}

In this section we present the structure of cotranslations, which are a specific kind of groupoid morphisms and give an algebraic counterpart to skew-products. This gives a generalization of group morphisms, which are known to describe dynamical systems.

\begin{definition} \cite[Definition 1.2]{Williams}
    We say that a set $\Xi$, doted of a subset $\Xi^{(2)}\subset \Xi\rimes \Xi$ (called the collection of composible pairs) and two maps $\bullet:\Xi^{(2)}\to \Xi$ given by $(\eta,\xi)\mapsto\eta\bullet\xi$ (called composition law), and $\ninv:\Xi\to \Xi$, is a \textbf{groupoid} if the following conditions are verified
    \begin{itemize}
        \item [i)] (associativity) If $(\eta,\xi),\,(\xi, \zeta)\in \Xi^{(2)}$, then  $(\eta\bullet\xi,\zeta),\,(\eta,\xi\bullet\zeta)\in \Xi^{(2)}$ and $(\eta\bullet\xi)\bullet\zeta=\eta\bullet(\xi\bullet\zeta)$,
        \item [ii)] (involution) $\ninv(\ninv(\eta))=\eta$ for every $\eta\in \Xi$,
        \item [iii)] (identity) for every $\eta\in \Xi$, we have $\left(\eta,\ninv(\eta)\right)\in\Xi^{(2)}$ and $(\eta,\xi)\in\Xi^{(2)}$ implies that $\ninv(\eta)\bullet(\eta\bullet\xi)=\xi$ and $(\eta\bullet\xi)\bullet\ninv(\xi)=\eta$.
    \end{itemize}
    
    Furthermore, we define $\Xi^{(0)}:=\left\{\eta\in \Xi:\eta=\ninv(\eta)=\eta\bullet\eta\right\}$ and call it the \textbf{units space} of the groupoid.
\end{definition}

The most trivial example of a groupoid is a group $G$, where $\bullet$ is the composition law and $\ninv(g)=g^{-1}$. In this case $G^{(2)}=G^2$ and $G^{(0)}=\{e\}$. 

\begin{example}
  {\rm
   Consider a topological space $\mathfrak{T}$, the collection $\widetilde{\Xi}=\left\{\eta:[0,1]\to \mathfrak{T}:\,\eta\text{ is continuous}\right\}$ and the quotient  $\Xi=\widetilde{\Xi}/\sim$, where $\sim $ is the equivalence relation given by homotopies that fix start and end. Doting $\Xi$ of
   $$\Xi^{(2)}:=\left\{(\eta,\xi)\in \Xi^2:\eta(1)=\xi(0)\right\},$$
   and
   $$(\eta\bullet\xi)(t)=\left\{ \begin{array}{lcc}
             \eta(2t) &  \text{ if } &   0\leq t\leq 1/2 \\
             \\ \xi(2t-1)& \text{ if } & 1/2\leq t\leq1
             \end{array}
             \right.,\qquad \ninv(\eta)(t)=\eta(1-t),$$
   it is easily deduced that $\Xi$ is a groupoid. This is usually called the \textbf{paths groupoid} with the operation of \textbf{concatenation}. In this case, the units space corresponds to the collection of the homotopy classes of constant paths.
    }
\end{example}

\begin{example}%\label{177}
{\rm
If a group $G$ acts by the left on a set $M$, the product $M\rimes G$ has groupoid structure. Indeed, setting
$$\left(M\rimes G\right)^{(2)}:=\left\{\left((x,g),(y,h)\right)\in \left(G\rimes M\right)^{2}:x=h\cdot y\right\},$$
and
$$\bullet\left((x,g),(y,h)\right)=(y,gh),\quad\ninv(x,g)=( g\cdot x,g^{-1}),$$
the groupoid axioms are easily followed. In this case, the units space corresponds to the collection of points of the form $(x,e)$, where $e$ is the group unit. Analogously we can define a groupoid for right actions.

\smallskip
This example is particularly useful when $G$ acts on itself by left translations, in which case we call $G\rimes G$ the \textbf{left translations groupoid} for $G$. Note that the groupoid $G\rimes G$ is never commutative, even if $G$ is Abelian, since $\left((g,h),(k,l)\right)\in \left(G\rimes G\right)^{(2)}$ does not imply $\left((k,l),(g,h)\right)\in \left(G\rimes G\right)^{(2)}$.
}
\end{example}

\begin{definition}  \cite[Definition 1.8]{Williams}
    If $\Xi$ and $\Upsilon$ are groupoids with composition laws $\bullet$ and $\star$ respectively, a \textbf{groupoid morphism} is a map $\vartheta:\Xi\to \Upsilon$ verifying that $(\eta,\xi)\in \Xi^{(2)}$, implies $\left(\vartheta(\eta),\vartheta(\xi)\right)\in \Upsilon^{(2)}$ and $\vartheta(\eta\bullet\xi)=\vartheta(\eta)\star\vartheta(\xi)$.
\end{definition}

Two basic facts about groupoid morphisms are that they preserve involutions and units spaces. In other words, for a groupoid morphism $\vartheta:\Xi\to\Upsilon$, $\ninv(\vartheta(\xi))=\vartheta(\ninv(\xi)$ for all $\xi\in \Xi$ and $\vartheta(\eta)\in \Upsilon^{(0)}$ for all $\eta\in \Xi^{(0)}$.

\begin{example}\label{159}
{\rm    Consider an element $X$ on some category $\mathscr{C}$. Set a group $G$,  $\gamma:G\to \mathbb{A}_X$ a group morphism and give $G\rimes G$ the left translations groupoid structure. By setting $Z:G\rimes G\to \mathbb{A}_X$ by $Z(g,h)=\gamma(h)$ we obtain a groupoid morphism. Indeed:
$$Z(g,kh)=\gamma(kh)=\gamma(k)\gamma(h)=Z(hg,k)Z(g,h).$$
}
\end{example}

    This example shows that every dynamical system is in particular given by a groupoid morphism, since a dynamical system is always given by a group morphism $\gamma:G\to \mathbb{A}_X$. The next example shows however, that groupoid morphisms describe a more general kind of dynamics.

\begin{example}\label{191}
    {\rm Consider the left translations groupoid structure on $\mathbb{R}\rimes\mathbb{R}$. Take $X=\mathbb{R}^d$ as a topological space, thus $\mathbb{A}_X$ is the group of all of its homeomorphisms. Consider a nonautonomous real differential equation $\dot{x}=F(t,x)$ such that for every $\xi\in \mathbb{R}^d$ and every $r\in \mathbb{R}$, there is a unique and globally defined solution $x_{r,\xi}:\mathbb{R}\to\mathbb{R}^d$ such that  $x_{r,\xi}(r)=\xi$. Then the map $Z:\mathbb{R}\rimes\mathbb{R}\to \mathbb{A}_X$ given by $\left[Z(r,t)\right](\xi)=x_{r,\xi}(t+r)$ is a groupoid morphism. Indeed:
    $$\left[Z(r,t+s)\right](\xi)=x_{r,\xi}(t+s+r)=x_{s+r,x_{r,\xi}(s+r)}(t+s+r)=\left[Z(s+r,t)\circ Z(r,s)\right](\xi),$$
where the second equality is a well known fact deduced from the 
uniqueness of solutions. Moreover, giving $\mathbb{A}_X$ the compact-open topology and the groupoid $\mathbb{R}\rimes \mathbb{R}$ the product topology, $Z$ turns out to be continuous.
    
    }
\end{example}

Given the dynamical relevance shown by the previous examples, we give these morphisms a distinctive name.

\begin{definition}
    Consider a group $G$ and an object $X$ on a category. A \textbf{cotranslation} is a groupoid morphism $Z:G\rimes G\to \mathbb{A}_X$, where $G\rimes G$ has the left translations groupoid structure.
    \end{definition}

Now we present our main theorem. Although we state it for the topological case, it is easily generalized for objects and skew-products on different categories.

\begin{theorem}
    Let $X$ be a locally compact Hausdorff topological space and $G$ a locally compact Hausdorff topological group. Give all function spaces, including $\mathbb{A}_X$, the compact-open topology. There is a bijective correspondence between skew-product dynamical systems $(X,G,Y,\sigma)$, where the action $\sigma$ is transitive, and continuous cotranslations $Z:G\rimes G\to \mathbb{A}_X$.
\end{theorem}

\begin{proof}
 Let $(X,G,Y,\sigma)$ be a skew-product dynamical system with $\sigma$ being transitive. The admissibility condition implies that $Y_G\subset \mathbb{A}_X$. On the other hand, as $\sigma$ is transitive, basic theory of group actions states that we can identify $Y$ with the quotient group,  $G/\stab(\sigma)$, where 
    $$\stab(\sigma)=\left\{g\in G: \sigma(g,y)=y,\,\forall\, y\in Y\right\}.$$

    Hence, we write $Y=\left\{\psi_{\overline{g}}:g\in G\right\}$, where $\overline{g}$ denotes the class of $g$ on the quotient $G/\stab(\sigma)$. Moreover, the left action $\sigma:G\rimes Y\to Y$ is rewritten as the left action $\hat{\sigma}:G\rimes Y\to Y$
given by $\hat{\sigma}(\psi_{h,\overline{g}})=\psi_{\overline{hg}}$, that is, it is identified with the action of left translations of $G$ on the quotient $G/\stab(\sigma)$.

    \smallskip
Now, we can define the continuous map $Z:G\rimes G\to \mathbb{A}_X$ given by$$Z(g,h)=\widetilde{\e}_G(\psi_{\overline{g}},h)=\psi_{\overline{g}}(h,\cdot),$$
 and we have
    \begin{eqnarray*}
        Z(g,kh)=Z(hg,k)\circ Z(g,h)&\Leftrightarrow&\left[Z(g,kh)\right](x)=\left[Z(hg,k)\circ Z(g,h)\right](x),\quad\forall\,x\in X\\
        &\Leftrightarrow&\left[\widetilde{\e}_G(\psi_{\overline{g}},kh)\right](x)=\widetilde{\e}_G(\psi_{\overline{hg}},k)\left[\left[ \widetilde{\e}_G(\psi_{\overline{g}},h)\right](x)\right],\quad\forall\,x\in X\\
        &\Leftrightarrow&\psi_{\overline{g}}(kh,x)=\psi_{\overline{hg}}\left(k,\psi_{\overline{g}}(h,x)\right),\quad\forall\,x\in X\\
        &\Leftrightarrow&\psi_{\overline{g}}(kh,x)=\left[\hat{\sigma}(h,\psi_{\overline{g}})\right]\left(k,\psi_{\overline{g}}(h,x)\right),\quad\forall\,x\in X,
    \end{eqnarray*}
and as the last condition is guaranteed by third the axiom of skew-products (Definition \ref{300}), then $Z$ is indeed a cotranslation.

\smallskip
Conversely, given a continuous cotranslation $Z:G\rimes G\to \mathbb{A}_X$, for each $g\in G$ we define $\psi_g:G\rimes X\to X$ given by  $\psi_g(h,x)=\left[Z(g,h)\right](x)$. Then, defining $Y:=\left\{\psi_g:g\in G\right\}$, it is admissible, since groupoid morphisms preserve units, and the only unit in $\mathbb{A}_X$ is the identity. On the other hand, defining the action $\sigma:G\times Y\to Y$ given by $\sigma(h,\psi_g)=\psi_{hg}$, or equivalently
    $$\left[\sigma(h,\psi_g)\right](k,x)=\left[Z(hg,k)\right](x),$$
    we obtain, by the same previous argument, that $(X,G,Y,\sigma)$ is a skew-product dynamical system, where the action $\sigma$ clearly results transitive.
\end{proof}

\begin{remark}\label{200}
{\rm
    The previous theorem illustrates that all the generalizations of dynamical systems we can obtain from skew-products, are also covered by cotranslations. However, the virtue of the latter is that they more clearly represent the algebraic structure behind these dynamics, just as group morphisms represent the algebraic structure of group actions.}
\end{remark}

\begin{comment}
    
\textcolor{red}{For this, we think this kind of maps deserve a distinctive name. Once again, we state it for the topological case, but it is easily generalized to other categories.}

\begin{definition}
\textcolor{red}{    Consider a topological group $G$ and a topological space $X$, both Hausdorff locally compact. Set on $G\times G$ the left translations groupoid structure and on $\mathbb{A}_X$ the compact-open topology. A (topological) \textbf{cotranslation} or \textbf{nonautonomous translation} or \textbf{nonautonomous dynamics} is a continuous groupoid morphism $Z:G\rimes G\to \mathbb{A}_X$, i.e.}
    $$Z(g,kh)=Z(hg,k)\circ Z(g,h)$$
\end{definition}

\textcolor{red}{Of course, we may define cotranslations on different categories. For instance, is $X$ is a smooth manifold and $G$ is a Lie group, $\mathbb{A}_X$ represents the Lie group of its diffeomorphisms and in order to have a {\it smooth} cotranslation, we ask the groupoid morphism $Z$ to be also smooth.}

\smallskip
\end{comment}

Now we will give further properties and applications for cotranslations. In the following, we fix a group $G$ and an element $X$ on some category $\mathscr{C}$.

\begin{notation}
    { \rm For a cotranslation $Z:G\rimes G\to \mathbb{A}_X$ we set  $Z^\inv:G\rimes G\to \mathbb{A}_X$, the map given by $Z^{\inv}(g,h)=\left[Z(g,h)\right]^{-1}$. We do not use $Z^{-1}$ in order to avoid confusion with a possible inverse function.
    }
\end{notation}

\begin{lemma}%\label{148}
    Set a cotranslation $Z:G\rimes G\to \mathbb{A}_X$. If $e\in G$ is the group unit, then for every $g,h\in G$ one has $$Z(g,e)=\Id\quad\text{and}\quad Z^\inv(g,h)=Z(hg,h^{-1}).$$
\end{lemma}

\begin{proof}
    The first equality follows again from the fact that groupoid morphisms preserve units, and the only unit in $\mathbb{A}_X$ is the identity. The second equality follows from the fact that groupoid morphisms preserve involutions and we know the involutions in $G\rimes G$.
\end{proof}

%\begin{remark}
%    {\rm The previous lemma shows that the invertibility of $Z(g,h)$ depends heavily on the invertibility of $h$ in $G$. We could define a similar theory replacing the group by a semigroup, but in that case we lose the invertibility of $Z(g,h)$ and the properties we can deduce from it. 
%    }
%\end{remark}

\begin{proposition}
Let $Z:G\rimes G\to \mathbb{A}_X$ be a cotranslation. Let $\gamma:G\to \mathbb{A}_X$ be a group morphism such that $$\gamma(k)Z(g,h)=Z(g,h)\gamma(k),\qquad\forall\,g,k,h\in G,$$
then $W:G\rimes G\to \mathbb{A}_X$ given by $W(g,h)=Z(g,h)\gamma(h)$ is a cotranslation.
\end{proposition}
\begin{proof}
    It is enough to see that
    \begin{align*}
        W(g,kh)=Z(g,kh)\gamma(kh)&=Z(hg,k)Z(g,h)\gamma(k)\gamma(h)\\
    &=Z(hg,k)\gamma(k)Z(g,h)\gamma(h)\\
    &=W(hg,k)W(g,h).
    \end{align*}
\end{proof}

\begin{example}
    {\rm 
    %Consider once more on $\mathbb{R}\rimes\mathbb{R}$ with the left translations groupoid structure.
    Consider $X=\mathbb{R}^d$ as a Banach space, thus $\mathbb{B}_X$ is the Banach algebra of its linear continuous operators and $\mathbb{A}_X$ is the topological group of all of its homeomorphic isomorphisms, {\it i.e. }$GL_d(\mathbb{R})$. Consider a real linear nonautonomous differential equation $\dot{x}=A(t)x(t)$, where $t\mapsto A(t)$ is locally integrable. Take $\Phi:\mathbb{R}\rimes\mathbb{R}\to \mathbb{A}_X$ the transition matrix associated to this equation. We can see, as in Example \ref{191}, that $Z:\mathbb{R}\rimes\mathbb{R}\to \mathbb{A}_X$ given by $Z(r,t)=\Phi(r,t+r)$ defines a cotranslation.

    \smallskip
    Choose $\lambda\in \mathbb{R}$ and define $\gamma:\mathbb{R}\to \mathbb{A}_X$ by $\gamma(t)=e^{-\lambda t} \cdot\Id$. It is easily a group morphism which verifies the conditions of the previous proposition. Hence, $W:\mathbb{R}\times\mathbb{R}\to \mathbb{A}_X$ given by $W(r,t)=Z(r,t)\gamma(t)$ is a cotranslation. Moreover, in this case $W$ is once more the cotranslation associated to a linear differential equation, since it is obtained in the same fashion as $Z$, but regarding the shifted linear nonautonomous differential equation $\dot{x}=\left[A(t)-\lambda\cdot\Id\right]x(t)$.
    }
\end{example}

Before studying more properties for cotranslations, we will introduce a specific type of them in order to highlight the difference between this construction and classic dynamical systems. This will also give emphasis on the difference between autonomous and nonautonomous equations. This motivates the following definition:

\begin{definition}
   Given a cotranslation $Z:G\rimes G\to \mathbb{A}_X$, we say it is \textbf{autonomous} if $Z(g,h)=Z(k,h)$ for every $g,h,k\in G$.
\end{definition}

\begin{proposition}%\label{161}
    A cotranslation $Z:G\rimes G\to\mathbb{A}_X$ is autonomous if and only if there is a group morphism $\gamma:G\to \mathbb{A}_X$ such that $Z(g,h)=\gamma(h)$.
\end{proposition}

\begin{proof}
We saw on Example \ref{159} that if $\gamma:G\to \mathbb{A}_X$ is a group morphism and we define $Z:G\rimes G\to \mathbb{A}_X$ by $Z(g,h)=\gamma(h)$ we obtain a cotranslation, which is autonomous by construction.
    
    \smallskip
Conversely, if $Z$ is a cotranslation, define $\gamma:G\to \mathbb{A}_X$ by $\gamma(g)=Z(e,g)$. It is easy to see that $\gamma$ is a group morphism, since 
$$\gamma(kh)=Z(e,kh)=Z(h,k)Z(e,h)=Z(e,k)Z(e,h)=\gamma(k)\gamma(h).$$
\end{proof}

Note once more, as in Remark \ref{200}, how this proposition illustrates that cotranslations generalize the notion of dynamical systems on any given category (as for instance, those given by autonomous differential equations, hence the name), since those are always given by a group morphism $\gamma:G\to\mathbb{A}_X$,  but clearly there are much more cotranslations which are not autonomous.

\smallskip
The following proposition shows the direct relation of cotranslations with group $\mathbb{Z}$ and nonautonomous difference equations.

\begin{proposition}%\label{176}
Set $X$ a Banach space, thus $\mathbb{B}_X$ is the Banach algebra of linear continuous operators and $\mathbb{A}_X$ is the topological group of its homeomorphic isomorphisms. There is a bijective correspondence between cotranslations $Z:\mathbb{Z}\times\mathbb{Z}\to \mathbb{A}_X$ and nonautonomous linear difference equations $x(n+1)=A(n)x(n)$, with $\mathbb{Z}\ni n\mapsto A(n)\in \mathbb{A}_X$, where the correspondence is given by
    \begin{equation*}
    Z(n,m)= \left\{ \begin{array}{lcc}
             A(n+m-1)A(n+m-2)\cdots A(n) &  \text{ if } &   m>0 \\
             \\ \Id & \text{ if }& m=0 \\
             \\ A^{-1}(n+m)A^{-1}(n+m+1)\cdots A^{-1}(n-1) & \text{ if } & m<0.
             \end{array}
   \right.
\end{equation*}
\end{proposition}
\begin{proof}
By defining $Z$ as in the statement of the proposition, starting from the function $n\mapsto A(n)\in \mathbb{A}_X$, clearly we obtain a groupoid morphism. Conversely, if we start with a cotranslation it is enough to define $A(n)=Z(n,1)$ and we obtain the desired equation.
\end{proof}

In the next section we will show a continuous analogous to the previous proposition. We finish this section by giving a generalization of the preceding proposition for finitely generated groups (for which we do not give a proof, since it follows the same steps as before), but first we need an auxiliary definition.

\begin{definition}
    Let $G$ be a discrete group with $n$ generators $\{\xi_1,\dots,\xi_n\}$ verifying the set of relations $R$. For each word $p\in R$ there is a positive integer $|p|$ called \textbf{length of the word} and a map $j_p:\{1,\dots,|p|\}\to \{1,\dots,n\}$ such that 
    $$p=\prod_{k=1}^{|p|}\xi_{j_p(k)}=\xi_{j_p(|p|)}\cdot\xi_{j_p(|p|-1)}\,\cdots\,\xi_{j_p(2)}\cdot\xi_{j_p(1)}.$$

    Let $X$ be a Banach space. We say that a collection of maps $A_i:G\to\mathbb{A}_X$, $i=1,\dots,n$, \textbf{preserves relations} if for each $p\in R$ and $\eta\in G$ one has
   \begin{align*}
       \Id_X&=\prod_{k=1}^{|p|}A_{j_p(k)}\left(\left[\prod_{\ell=1}^{k-1}\xi_{j_p(\ell)}\right]\eta\right)\\
       &:=A_{j_p(|p|)}(\xi_{j_p(|p|-1)}\cdots\xi_{j_p(1)}\eta)\circ \cdots \circ A_{j_p(2)}(\xi_{j_p(1)}\eta)\,\circ\, A_{j_p(1)}(\eta).
   \end{align*} 
\end{definition}

\begin{proposition}
    Let $G$ be a discrete group with $n$ generators $\{\xi_1,\dots,\xi_n\}$ and $X$ an object on some category $\mathscr{C}$. A multivariable nonautonomous difference equation of $G$ on $X$ is an equation of the form
    $$x(\xi_i\eta)=A_i(\eta)x(\eta),$$
    where the collection of maps $A_i:G\to \mathbb{A}_X$, $i=1,\dots,n$ preserves relations. A solution to this equation is a map $x:G\to X$. There is a bijective correspondence between multivariable nonautonomous difference equations of $G$ on $X$ and cotranslations $Z:G\rimes G\to \mathbb{A}_X$, which is given by $$A_i(\eta)=Z\left(\eta,\xi_i\right).$$
\end{proposition}

\section{Differentiable cotranslations on Banach spaces}

In this section we study the differentiability of cotranslations and analyze their relation to differential equations. During this section we fix a Banach space $X$ over the field $\mathbb{K}$, which might be $\mathbb{R}$ or $\mathbb{C}$. We set
\begin{itemize}
\item $\mathbb{L}_X:=\mathbb{L}(X)$ the collection of all linear operators $T:X\to X$,
    \item $\mathbb{B}_X:=\mathbb{B}(X)$ the Banach algebra of all continuous elements of $\mathbb{L}_X$, given the operator norm,
    \item $\mathbb{A}_X:=\mathbb{A}(X)$ the group of all invertible elements of $\mathbb{B}_X$ whose inverse is also in $\mathbb{B}_X$, as a topological subspace of $\mathbb{B}_X$.
\end{itemize}

If $X$ has finite dimension $d$, then $\mathbb{L}_X=\mathbb{B}_X$ and $\mathbb{A}_X\cong GL_d(\mathbb{K})$.

\smallskip
The group for the cotranslations we study in this section is the field $\mathbb{K}$, but the formalism presented here allows generalizations of the results to other non commutative by Lie groups.

\smallskip
We use Banach spaces in this section because we want $\mathbb{B}_X$ to have the structure of a Banach algebra, thus we will be able to study derivatives for groupoid morphisms in the sense of G\^ateaux. However, it is easy to see that these constructions apply just as well in other categories that admit such sense of derivatives. For instance, an alternative is to consider the strong topology in $\mathbb{B}_X$ (as usually studied on \cite{Pazy}) instead of the operator norm.

\smallskip
It is worth noting that for groups it is well known that continuous morphisms between Lie groups are immediately smooth \cite[Problem 20-11]{Lee}, however, for cotranslations ({\it i.e.} groupoid morphisms) there are differences to be considered. 

\begin{notation}
    {\rm Consider $\varphi:\mathbb{K}\rimes\mathbb{K}\to \mathbb{B}_X$ and $\psi:\mathbb{K}\to \mathbb{B}_X$. We set the following notation for these limits (independently if they exist or not):
    \begin{itemize}
        \item [i)] $\partial_1\varphi (r,t)=\lim_{h\to 0}\frac{\varphi(r+h,t)-\varphi(r,t)}{h}$,
        \item [ii)] $\partial_2\varphi (r,t)=\lim_{h\to 0}\frac{\varphi(r,t+h)-\varphi(r,t)}{h}$,
        \item [iii)] $\frac{d}{du}\left[\psi(u)\right]=\lim_{h\to 0}\frac{\psi(u+h)-\psi(u)}{h}$.
    \end{itemize}
    
Notation iii) will be useful when we want to apply derivatives to functions obtained as composition of other functions or when we want to make an emphasis on the {\it variable} respect to which we want to derivate, while notations i) y ii) are meant to make emphasis on the {\it position} of the variable respect to which we want to derivate.
    }
\end{notation}

\begin{lemma}\label{153}
    Let $Z:\mathbb{K}\rimes \mathbb{K}\to \mathbb{A}_X$ be a continuous cotranslation. If for every $t\in \mathbb{K}$ the map $r\mapsto Z(r,t)$ is derivable at $r=r_0$, for some $r_0\in \mathbb{K}$, then the map $r\mapsto Z^{\inv}(r,t)$ is derivable at $r=r_0$ and
    $$\partial_1Z^\inv(r_0,t)=-Z^{\inv}(r_0,t)\Big[\partial_1Z(r_0,t)\Big] Z^{\inv}(r_0,t).$$
\end{lemma}
\begin{proof}
    By the following
    \begin{eqnarray*}
        Z(r,t) Z^\inv(r,t)=\Id&\Rightarrow&\frac{d}{dr}\left[Z(r,t)Z^\inv(r,t)\right]=0\\
        &\Leftrightarrow&\Big[\partial_1Z(r,t)\Big]\ Z^\inv(r,t)+Z(r,t) \left[\partial_1Z^\inv(r,t)\right]=0\\
        &\Leftrightarrow&\partial_1 Z^\inv(r,t)=-Z^\inv(r,t) \Big[\partial_1 Z(r,t)\Big] Z^\inv(r,t),
    \end{eqnarray*}
    it is easy to see that the derivatives $\partial_1$ at $r=r_0$ exist simultaneously for both $r\mapsto Z(r,t)$ and $r\mapsto Z^\inv(r,t)$.
\end{proof}

The following result is easily obtained with an analogous proof.

\begin{lemma}\label{171}
       Let $Z:\mathbb{K}\rimes \mathbb{K}\to \mathbb{A}_X$ be a continuous cotranslation. If for every $r\in \mathbb{K}$ the map $t\mapsto Z(r,t)$ is derivable at $t_0$, for some $t_0\in \mathbb{K}$, then the map $t\mapsto Z^\inv(r,t)$ is derivable at $t=t_0$ and
       $$\partial_2Z^\inv(r,t_0)=-Z^{\inv}(r,t_0)\Big[\partial_2Z(r,t_0)\Big] Z^{\inv}(r,t_0).$$
\end{lemma}

\begin{lemma}%\label{169}
    Let $Z:\mathbb{K}\rimes \mathbb{K}\to \mathbb{A}_X$ be a continuous cotranslation. Suppose that for every $t\in\mathbb{K}$ the map $r\mapsto Z(r,t)$ is derivable at $r=0$. Then, for every $t\in \mathbb{K}$ the function $r\mapsto Z(r,t)$ is derivable and
    $$\partial_1 Z(r,t)=\Big[\partial_1 Z(0,t+r)\Big] Z(r,-r)-Z(r,t)\Big[\partial_1Z(0,r)\Big] Z(r,-r).$$
\end{lemma}

\begin{proof}
We suppose the following limit exists
$$\lim_{h\to 0}\frac{Z(0+h,t)-Z(0,t)}{h}.$$

We have
\begin{eqnarray*}
    \frac{Z(r+h,t)-Z(r,t)}{h}&=&\frac{Z(h,t+r)Z^\inv(h,r)-Z(0,r+t)Z^\inv(0,r)}{h}\\
    \\&=&\frac{Z(h,t+r)-Z(0,r+t)}{h}Z^\inv(h,r)\\
   \\&&+Z(0,r+t)\frac{Z^\inv(h,r)-Z^\inv(0,r)}{h}.
\end{eqnarray*}

Both addends at the right hand side have limit when $h\to 0$ by hypothesis (and Lemma \ref{153}), thus the left hand side as limit as well, hence $r\mapsto Z(r,t)$ is derivable at every point for every fixed $t\in\mathbb{K}$ and
\begin{eqnarray*}
    \partial_1Z(r,t)&=&\Big[\partial_1Z(0,t+r)\Big]Z^\inv(0,r)+Z(0,r+t)\Big[\partial_1Z^{\inv}(0,r)\Big]\\
    &=&\Big[\partial_1Z(0,t+r)\Big]Z(r,-r)-Z(0,r+t)Z(r,-r)\Big[\partial_1Z(0,r)\Big]Z(r,-r)\\
    &=&\Big[\partial_1Z(0,t+r)\Big]Z(r,-r)-Z(r,t)\Big[\partial_1Z(0,r)\Big]Z(r,-r).
\end{eqnarray*}
\end{proof}

Analogously we have the following result.

\begin{lemma}\label{165}
    Let $Z:\mathbb{K}\rimes\mathbb{K}\to \mathbb{A}_X$ be a continuous cotranslation. Suppose that for every $r\in\mathbb{K}$ the function $t\mapsto Z(r,t)$ is derivable at $t=0$. Then, for every $r\in \mathbb{K}$ the function $t\mapsto Z(r,t)$ is derivable and
    $$\partial_2Z(r,t)=\Big[\partial_2Z(r+t,0)\Big]Z(r,t).$$
\end{lemma}

\begin{proof}
    Suppose that for every $r\in \mathbb{K}$ the following limit exists
    $$\lim_{h\to 0}\frac{Z(r,h)-Z(r,0)}{h}.$$

    We have
    \begin{eqnarray*}
        \frac{Z(r,t+h)-Z(r,t)}{h}&=&\frac{Z(r+t,h)Z(r,t)-Z(r,t)}{h}\\
        &=&\frac{Z(r+t,h)-\Id}{h}Z(r,t)\\
        &=&\frac{Z(r+t,h)-Z(r+t,0)}{h}Z(r,t).
    \end{eqnarray*}

By taking limits $h\to 0$ we obtain the desired identity.
\end{proof}

Now we show results which show how to deduce derivability respect to one coordinate when we have information about the other.

\begin{lemma}\label{147}
    Let $Z:\mathbb{K}\rimes\mathbb{K}\to \mathbb{A}_X$ be a continuous cotranslation such that for every $r\in \mathbb{K}$ the map $t\mapsto Z(r,t)$ is derivable. Then, the map $r\mapsto Z(r,t)$ is derivable for every $t\in \mathbb{K}$ and
    $$\partial_1Z(r,t)=\partial_2Z(r,t)-Z(r,t) \Big[\partial_2Z(r,0)\Big].$$
\end{lemma}

\begin{proof}
    Suppose that for every $r,t\in \mathbb{K}$ the following limit exists
    $$\lim_{h\to 0}\frac{Z(r,t+h)-Z(r,t)}{h}.$$
    
    Note that
     \begin{eqnarray*}
        \frac{Z(r,t+h)-Z(r,t)}{h}&=&\frac{Z(r+h,t)Z(r,h)-Z(r,t)}{h} \\
        \\&=&\frac{Z(r+h,t)Z(r,h)-Z(r,t)Z(r,h)}{h}+\frac{Z(r,t)Z(r,h)-Z(r,t)}{h} \\
        \\&=&\frac{Z(r+h,t)-Z(r,t)}{h}Z(r,h)+Z(r,t)\frac{Z(r,h)-Z(r,0)}{h}.
    \end{eqnarray*}
    from where, reorganizing terms and taking limits we obtain
    \begin{eqnarray*}
        \lim_{h\to 0}\frac{Z(r+h,t)-Z(r,t)}{h}&=&\Big[\partial_2Z(r,t)\Big] Z(r,0)-Z(r,t)\Big[ \partial_2Z(r,0)\Big]\\
        \\&=&\partial_2Z(r,t)-Z(r,t)\Big[\partial_2Z(r,0)\Big].
    \end{eqnarray*}
\end{proof}

We would like to give a reciprocal result to Lemma \ref{147}, but we have not yet been able to prove or find a counterexample for this. As a summary of the previous lemmas we state the following:

\begin{corollary}
    For a continuous cotranslation $Z: \mathbb{K}\rimes\mathbb{K}\to \mathbb{A}_X$, the following statements are equivalent:
    \begin{itemize}
        \item [i)] $t\mapsto Z(r,t)$ is derivable at $t=0$, for every $r\in \mathbb{K}$.
        \item [ii)] $t\mapsto Z(r,t)$ is derivable for every $r\in \mathbb{K}$.
       % \item [iii)] $Z$ is differentiable.
    \end{itemize}

    Moreover, each one of them implies the following statements, which are equivalent:
    \begin{itemize}
        \item [iii)] $r\mapsto Z(r,t)$ is derivable at $r=0$, for every $t\in \mathbb{K}$.
        \item [iv)] $r\mapsto Z(r,t)$ is derivable for every $t\in \mathbb{K}$.
    \end{itemize}
\end{corollary}

Ideally we would show that a cotranslation which is (separately) derivable respect to both variables is differentiable as a two-variable function. However, as the following lemma shows, an extra condition is needed for this.

\begin{lemma}%\label{168}
Let $Z:\mathbb{K}\rimes\mathbb{K}\to \mathbb{A}_X$ a continuous cotranslation. Suppose that for every $r\in \mathbb{K}$ the map $t\mapsto Z(r,t)$ is derivable. If the mapping
$$(r,t)\mapsto \partial_2 Z(r,t)\,,$$
is continuous, then $Z$ is differentiable as a two-variable function.
\end{lemma}

\begin{proof}
    From Lemma \ref{147} we know that as for every $r\in \mathbb{K}$ the map $t\mapsto Z(r,t)$ is derivable, then  $r\mapsto Z(r,t)$ is derivable for every $t\in \mathbb{R}$. Hence, for $h_1,h_2\in \mathbb{K}$, with $h_2\neq 0$ and $h_1+h_2\neq 0$, we have
    \begin{align*}
        \frac{Z(r+h_1,t+h_2)-Z(r,t)}{|h_1|+|h_2|}=&\frac{Z(r+h_1+h_2,t)-Z(r,t)}{|h_1|+|h_2|}Z(r+h_1,h_2)\\
        &+Z(r,t)\frac{Z(r+h_1,h_2)-\Id}{|h_1|+|h_2|}\\
        =&\frac{h_1+h_2}{|h_1|+|h_2|}\frac{Z(r+h_1+h_2,t)-Z(r,t)}{h_1+h_2}Z(r+h_1,h_2)\\
        &+Z(r,t)\frac{Z(r+h_1,h_2)-Z(r+h_1,0)}{h_2}\frac{h_2}{|h_1|+|h_2|}.
    \end{align*}

All elements at the right hand side on the last equality have limit when $(h_1,h_2)\to 0$, since $Z$ is derivable respect to each of its variables and $(r,t)\mapsto \partial_2Z(r,t)$ is continuous. If either $h_2=0$ or $h_1+h_2=0$, some terms nullify at the second equality.
\end{proof}

\begin{remark}
    {\rm Given the identity found on Lemma \ref{165}, in order to satisfy the continuity of $(r,t)\mapsto \partial_2Z(r,t)$ required on the previous lemma it is enough to verify that $r\mapsto \partial_2Z(r,0)$ is continuous.}
\end{remark}

In the following we study the relation of cotranslations and nonautonomous differential equations. As we want to consider more general spaces than $\mathbb{R}^d$, let us consider the following definition:

\begin{definition}
A linear nonautonomous differential equation on $X$ is an equation of the form  
    $$\frac{d x}{dt}=A(t)x(t),$$
    where $t\mapsto A(t)\in \mathbb{L}_X$, and its solutions are differentiable maps $x:\mathbb{K}\to X$.
\end{definition}

Even in the autonomous case ({\it i.e.}, $t\mapsto A(t)$ is constant) the existence of solutions of these kind of equations can be a hard problem when $X$ is an infinite dimensional space. On the other hand, an usual generalization is to look for solutions only partially defined, for instance with domain $[0,\infty)\subset \mathbb{R}$. Moreover, on infinite dimensional spaces it may be interesting to study such equations where $A$ is not globally defined, but only on a dense subspace, in which case it can be difficult enough to find solutions defined on a compact interval like $[0,t]$, in contrast to the solutions defined on the whole group $\mathbb{K}$ as we propose. For more details we refer the reader to  \cite[Chapter 4]{Pazy}). 

\smallskip
On the other hand, the nonautonomous case presents even greater difficulties. On finite dimension, it is well known that it is enough to ask the function $t\mapsto A(t)$ to be locally integrable in order to guarantee the existence and 
uniqueness of solutions, {\it i.e.}, the existence of an evolution matrix. On infinite dimensional spaces, the problem for the existence of globally defined solutions is much harder. A partial result stays that if $t\mapsto A(t)\in \mathbb{B}_X$ is continuous under the uniform operator norm, then we have the existence and 
uniqueness of solutions defined on a bounded interval \cite[Theorem 5.1.1]{Pazy}.

\smallskip
In general, the problem of existence and 
uniqueness of globally defined solutions has not been solved, but we only know specific conditions under which we have this, as hyperbolicity or parabolicity (for more details we refer the reader to \cite[Chapter 5]{Pazy}). We dedicate the end of this section to give an answer to this problem using the structure of cotranslations.

\begin{proposition}\label{178}
Let $Z:\mathbb{K}\rimes\mathbb{K}\to \mathbb{A}_X$ a continuous cotranslation. Suppose the map $t\mapsto Z(r,t)$ is derivable for every $r\in \mathbb{R}$. Define $A:\mathbb{K}\to \mathbb{L}_X$ and the operator $\Psi:\mathbb{K}^2\to \mathbb{A}_X$ by
$$A(u):=\partial_2Z(u,0),\qquad\Psi(u,v):=Z(v,u-v),$$ 
    then, the following are verified:
    \begin{itemize}
    \item [i)] $\Psi(u,v)\Psi(v,w)=\Psi(u,w)$ for every $u,v,w\in \mathbb{K}$,
        \item [ii)] $\frac{d\Psi}{du}=A(u)\Psi(u,v)$,
        \item [iii)] $\frac{d\Psi}{dv}=-\Psi(u,v)A(v)$,
        \item [iv)] for every $\xi\in X$, the map $\psi_{v,\xi}:\mathbb{K}\to X$, given by $\psi_{v,\xi}(u)=\left[\Psi(u,v)\right](\xi)$ verifies $\psi_{v,\xi}(v)=\xi$ and is a solution to the equation 
    \begin{equation}\label{170}
    \frac{dx}{du}=A(u)x(u).
    \end{equation}
    \end{itemize}
\end{proposition}

\begin{proof}
It is easy to see that each $A(u)$ is a linear transformation of $X$ (however, we cannot ensure in general that it is continuous). Similarly, we cannot in general state that $u\mapsto A(u)$ is continuous (even when $\mathbb{L}_X$ has a topology that extends that of $\mathbb{B}_X$). To verify {\it i)} it is enough to see:
$$\Psi(u,v)\Psi(v,w)=Z(v,u-v)Z(w,v-w)=Z\left(w,(v-w)+(u-v)\right)=Z(w,u-w)=\Psi(u,w).$$

On the other hand, on Lemma \ref{165} we proved the identity 
$$\partial_2Z(r,t)=\Big[\partial_2Z(r+t,0)\Big]Z(r,t),$$
hence
\begin{equation*}
 \frac{d\Psi}{du}(u,v)=\frac{d}{du}\left[Z(v,u-v)\right]=\partial_2Z(v,u-v)=\Big[\partial_2Z(v+u-v,0)\Big]Z(v,u-v)=A(u)\Psi(u,v),
\end{equation*}
from where {\it ii)} follows. Then, trivially $\psi_{v,\xi}$ is a solution to (\ref{170}) and
$$\psi_{v,\xi}(v)=\left[\Psi(v,v)\right](\xi)=\left[Z(v,v-v)\right](\xi)=\Id_X(\xi)=\xi,$$ 
thus verifying {\it iv)}. Finally, note that
$$\Psi(u,v)=Z(v,u-v)=Z(0,u)Z(v,-v)=Z(0,u)Z^\inv(0,v),$$
hence, using the identity from Lemma \ref{171} we obtain:
\begin{align*}
    \frac{d\Psi}{dv}(u,v)=Z(0,u)\frac{d}{dv}\left[Z^\inv(0,v)\right]&=Z(0,u)\Big[\partial_2Z^\inv(0,v)\Big]\\
    &=-Z(0,u)Z^\inv(0,v)\Big[\partial_2Z(0,v)\Big]Z^\inv(0,v)\\
    &=-\Psi(u,v)\Big[\partial_2Z(v,0)\Big]\\
    &=-\Psi(u,v)A(v),
\end{align*}
where the second to last equality follows from the identify of Lemma \ref{165}, thus proving {\it iii)}.
\end{proof}
 
The existence of a function with the properties of $\Psi$ on the previous proposition is exactly what is needed to describe globally defined solutions, for every initial condition, to a linear differential equation. It is easy to see that they are  generalization of the well known transition matrices on finite dimension. To formalize this notion we state the following definition, adapted from \cite[Definition 5.1.3]{Pazy}.

\begin{definition}
    A map $\Psi:\mathbb{K}\rimes\mathbb
    {K}\to \mathbb{B}_X$ is an \textbf{evolution operator} (or evolution system) if the following conditions are verified
    \begin{itemize}
        \item [i)] $\Psi(r,r)=\Id$, $\Psi(r,t)\Psi(t,s)=\Psi(r,s)$ for every $r,s,t\in \mathbb{K}$,
        \item [ii)] $(r,t)\to \Psi(r,t)\in\mathbb{B}_X$ is strongly continuous, i.e., $(r,t)\to \left[\Psi(r,t)\right](x)\in X$ is continuous for every $x\in X$.
    \end{itemize}

    If furthermore, there is a map $t\mapsto A(t)\in\mathbb{L}_X$ such that
    \begin{itemize}
        \item [iii)] $\partial_1\Psi(r,t)=A(r)\Psi(r,t)$ y $\partial_2\Psi(r,t)=-\Psi(r,t)A(t)$, 
    \end{itemize}
    we say that $\Psi$ is the \textbf{evolution operator associated} to the differential equation $\dot{x}=A(t)x(t)$.
\end{definition}

On \cite[Theorem 4.1.3]{Pazy} it is proved that a real autonomous linear differential equation $\dot{x}=Ax$ on a Banach space has a solution defined on $[0,\infty)$ and uniquely defined for every initial condition on $X$ if and only if $A$ is a linear operator, defined on a dense subspace of $X$ which is obtained as the strong derivative at the origin of a semigroup morphism $T:[0,\infty)\to\mathbb{B}_X$. Inspired by this, we present the following definition and theorem.

\begin{definition}
    For a differentiable cotranslation $Z:\mathbb{K}\rimes\mathbb{K}\to \mathbb{A}_X$, its \textbf{infinitesimal generator} is the function $A:\mathbb{K}\to \mathbb{L}_X$ given by
    $A(t)=\partial_2Z(t,0)$. 
\end{definition}

Note that the infinitesimal generator of a cotranslation corresponds to the derivative respect to the second coordinate evaluated on the unit space $\left(\mathbb{K}\rimes\mathbb{K}\right)^{(0)}$ of the groupoid.

\begin{theorem}
    A linear nonautonomous differential equation on a Banach space $X$
    \begin{equation}\label{179}
        \dot{x}(t)=A(t)x(t),
    \end{equation}
    has an evolution operator associated to it (i.e. unique and globally defined solution for every initial condition on $X$) if and only if $A$ is the infinitesimal generator of a differentiable cotranslation $Z:\mathbb{K}\rimes \mathbb{K}\to \mathbb{A}_X$.
\end{theorem}

\begin{proof}
    Proposition \ref{178} states that if $A(t)=\partial_2Z(t,0)$ for some cotranslation, then there is an evolution operator associated to (\ref{179}). 
    
    \smallskip
    On the other hand, if (\ref{179}) has an evolution operator associated, say $\Psi$, then set $Z:\mathbb{K}\rimes\mathbb{K}\to\mathbb{A}_X$  by $Z(r,t)=\Psi(t+r,r)$. It is easy to see that such $Z$ is a differentiable groupoid morphism and furthermore, by Lemma \ref{165}, we conclude $A(t)=\partial_2Z(t,0)$. 
\end{proof}

\section{Partial cotranslations}
In this section we consider a partial version of cotranslations. We fix a group $G$ and consider in $G\rimes G$ the left translations groupoid structure. When we consider a topology in $G$, we assume it to be Hausdorff and locally compact and in $G\rimes G$ we consider the product topology. We also consider a Banach space $X$ over the field $\mathbb{K}$. Some of the results we present in this section apply for general Banach spaces, but the main conclusions need $X$ to be finite dimensional. Moreover, at the end of the section we need the notion of orthonormal basis, thus $X$ will need to be also a Hilbert space. In other words, our main results hold when we consider the space $X=\mathbb{K}^d$, with its Euclidean norm. In this case we obtain $\mathbb{B}_X=\mathcal{M}_d(\mathbb{K})$, the space of square $d\rimes d$ matrices, and $\mathbb{A}_X=GL_d(\mathbb{K})$ the general linear group of dimension $d$.

\begin{definition} 
    \begin{itemize}
        \item [a)] We say a function  $W:G\times G\to \mathbb{B}_X$ is a \textbf{partial cotranslation}  if
        $$W(g,kh)=W(hg,k)\circ  W(g,h),\quad\text{ for all } g,h,k\in G.$$

        \item [b)] For two partial cotranslations $W,V:G\rimes G\to\mathbb{B}_x$, we say they are \textbf{mutually orthogonal} if $W(hg,k)V(g,h)=V(hg,k)W(g,h)=0$ for all $g,h,k\in G$.
    \end{itemize}
\end{definition}

Unlike a cotranslation, a partial cotranslation is not necessarily a groupoid morphism, since its codomain is not a groupoid (at least not if we define $\mathbb{B}_X^{(2)}=\mathbb{B}_X^2$).

\begin{lemma}\label{331}
    If $W:G\rimes G\to \mathbb{B}_X$ is a partial cotranslation, then $W(g,e)$ is idempotent for all $g\in G$.
\end{lemma}

\begin{proof}
  By definition, we have $W(g,e)W(g,e)=W(g,e\cdot e)=W(g,e)$ for all $g\in G$.
\end{proof}

\begin{lemma}\label{352}
    Let $W:G\rimes G\to \mathbb{B}_X$ be a partial cotranslation. For all $g,h,k\in G$ one gets $\ker W(h,g)=\ker W(h,k)$. In particular, if $X$ is finite dimensional, then $\rank W(h,g)=\rank W(h,k)$.
\end{lemma}

\begin{proof}
    From the identity
    $$W(h,k)=W(gh,kg^{-1})W(h,g)$$
    one gets $\xi\in \ker W(h,g)\Rightarrow \xi\in \ker W(h,k)$, thus $\ker W(h,g)\leq \ker W(h,k)$ and by symmetry of the argument we obtain the equality. The second conclusion is trivial if $X$ is finite dimensional.
\end{proof}

A basic fact of linear algebra we need is that for two transformations $A,B\in \mathbb{B}_X$, if $AB$ has the same kernel as $B$, then $\dim \ker A\leq \dim\ker B$.

\begin{proposition}\label{366}
      Let $W:G\rimes G\to \mathbb{B}_X$ be a partial cotranslation, where $X$ is finite dimensional. Then $\rank W(g,h)=\rank W(k,\ell)$ for every $g,h,k,\ell\in G$. 
\end{proposition}

\begin{proof}
    Starting from the identity
    $$W(hg,h^{-1})W(g,h)=W(g,e),$$
    as Lemma \ref{352} indicates $W(g,h)$ and $W(g,e)$ have the same kernel, then the previous observation implies $\dim\ker W(hg,h^{-1})\leq \dim\ker W(g,h)$. 

\smallskip
    Now, from the identity
    $$W(g,h)W(hg,h^{-1})=W(hg,e),$$
    as Lemma \ref{352} showed that $W(hg,h^{-1})$ and $W(hg,e)$ have the same kernel, then $\dim\ker W(g,h)\leq \dim \ker W(hg,h^{-1})$. In conclusion, $\dim\ker W(hg,h^{-1})=\dim\ker W(g,h)$, thus $\rank W(hg,h^{-1})=\rank W(g,h)$.

    \smallskip
    Taking $h=kg^{-1}$ in the last equality we obtain $\rank W(k,gk^{-1})=\rank W(g,kg^{-1})$, but for arbitrary $h,\ell\in G$ we know, from Lemma \ref{352} that $\rank W(g,kg^{-1})=\rank W(g,h)$ and $\rank W(k,gk^{-1})=\rank W(k,\ell)$, thus $\rank W(g,h)=\rank W(k,\ell)$.
\end{proof}

In view of the above, we can define without ambiguity:

\begin{definition}\label{365}
     Let $X$ be finite dimensional. For a partial cotranslation $W:G\rimes G\to\mathbb{B}_X$, the \textbf{rank} of the partial cotranslation is $\rank \,W:=\rank\, W(g,h)$ for arbitrary $g,h\in G$.
\end{definition}

As every cotranslation is in particular a partial cotranslation, the following definition applies as well for cotranslations.

\begin{definition}\label{158}
    Given a partial cotranslation $V:G\rimes G\to \mathbb{B}_X$, we say a function $\P:G\to \mathbb{B}_X$ is:
    \begin{itemize}
        \item [a)] \textbf{projector} if $\P(g)$ is idempotent on $\mathbb{B}_X$ for every $g\in G$.
        \item [b)] \textbf{invariant projector} associated to $V$ if it is a projector and
        $$\P(hg)V(g,h)=V(g,h)\P(g),\qquad\forall\,g,h\in G.$$
        \item [c)] For other invariant projectors $\mathrm{Q}:G\to \mathbb{B}_X$, we say $\P$ and $\mathrm{Q}$ are \textbf{mutually orthogonal} if $\P(g)\mathrm{Q}(g)=\mathrm{Q}(g)\P(g)=0$ for every $g\in G$.
    \end{itemize}
\end{definition}

It is easy to see that if $\P$ is an invariant projector for a partial cotranslation, then $\Id-\P$ is invariant as well, and they are mutually orthogonal.

\begin{lemma}%\label{187}
    If $W:G\rimes G\to \mathbb{B}_X$ is a partial cotranslation and we define $\P:G\to \mathbb{B}_X$ by $\P(g)=W(g,e)$, called its \textbf{projector of the units space}, we obtain that $\P$ is an invariant projector associated to $W$. If $W$ is continuous, then its projector of the units space is continuous as well.
\end{lemma}

\begin{proof}
    We know from Lemma \ref{331} that it is indeed a projector. To obtain its invariance it is enough to see
    \begin{align*}
        W(g,h)\P(g)=W(g,h)W(g,e)=W(g,h)=W(hg,e)W(g,h)=\P(hg)W(g,h).
    \end{align*}

The last statement is trivial.
\end{proof}

\begin{lemma}\label{188}
 Let $X$ be finite dimensional.  If $\P$ is an invariant projector associated to a (non-partial) cotranslation $Z:G\rimes G\to \mathbb{A}_X$, then $\rank\,\P(g)$ is a constant (for all $g\in G$). Moreover, if $Z$ is continuous, then $g\mapsto \P(g)$ is continuous as well.
\end{lemma}

\begin{proof}
By definition we have
$$\P(e)Z(g,g^{-1})=Z(g,g^{-1})\P(g),$$
thus $\P(g)=Z^\inv (g,g^{-1})\P(e)Z(g,g^{-1})=Z(e,g)\P(e)Z(g,g^{-1})$, which shows that $\rank \P(g)=\rank\P(e)$ for all $g\in G$. If $Z$ is continuous, this identity also shows $\P$ is continuous.
\end{proof}

The previous result is not true for partial cotranslations. Consider for instance $X=\mathbb{R}^2$, $G=\mathbb{Z}$, and the partial cotranslation given by $W(m,n)=\begin{pmatrix}
0&0\\
0&2^n
\end{pmatrix}$. The projector given by
$\P(n)=\begin{pmatrix}
0&0\\
0&1
\end{pmatrix}$ for even $n$ and $\P(n)=\Id$ for odd $n$, is clearly invariant, but its rank is not constant. 

\smallskip

Nevertheless, the projector of the units space of a partial cotranslation does indeed have constant rank.

\begin{proposition}\label{172}
    If $V:G\times G\to \mathbb{B}_X$ is a partial cotranslation and $\P$ is an invariant associated projector, then $W:G\rimes G\to \mathbb{B}_X$ given by $W(g,h)=V(g,h)\P(g)$ is a partial cotranslation.
    
    \smallskip
    If moreover both $\P$ and $V$ are continuous, then $W$ is continuous as well.
\end{proposition}

\begin{proof}
It is enough to see that
    \begin{eqnarray*}
        W(g,kh)=V(g,kh)\P(g)&=&V(g,kh)\P(g)^2\\
    &=&V(hg,k)V(g,h)\P(g)^2\\
    &=&V(hg,k)\P(hg)V(g,h)\P(g)\\
    &=&W(hg,k)W(g,h).
    \end{eqnarray*}

    Finally, continuity follows trivially.
\end{proof}

\begin{example}\label{417}
    {\rm Consider the linear nonautonomous  differential equation $\dot{x}=A(t)x$ with transition matrix $\Phi$. We know from Example \ref{191} that $Z(s,t)=\Phi(t+s,s)$ defines a cotranslation with $X=\mathbb{R}^d$ and $G=\mathbb{R}$. Consider now an invariant projector $\P$ for this cotranslation (in the sense of Definition \ref{158}). By Lemma \ref{188}, $\P$ is continuous. Proposition \ref{172} shows that $W(s,t)=Z(s,t)\P(s)$ defines a continuous partial cotranslation. 
    
    \smallskip
Moreover, in this case $\P$ is also an invariant projector for this differential equation, in the usual sense for nonautonomous equations \cite[p. 246]{Siegmund}. Indeed:
    \begin{equation*}
        \P(t)\Phi(t,s)=\P(t)Z(s,t-s)=\P\left((t-s)+s\right)Z(s,t-s)=Z(s,t-s)\P(s)=\Phi(t,s)\P(s).
    \end{equation*}
    
Note that $W$ can also be written $W(s,t)=\Phi(t+s,s)\P(s)$, which is a matrix function that describes the solutions to the equation $\dot{x}=A(t)x$ whose graph lies in the linear integral manifold \cite[Definition 2.2]{Siegmund} associated to the image of this projector \cite[Lemma 2.1]{Siegmund}.}
\end{example}

\begin{lemma}\label{372}
    If $W,V:G\rimes G\to \mathbb{B}_X$ are mutually orthogonal partial cotranslations, then $W+V:G\rimes G\to \mathbb{B}_X$ is a partial cotranslation. If both $W$ and $V$ are continuous, then $W+V$ is as well. 
\end{lemma}

\begin{proof}
Note that
        \begin{eqnarray*}
        W(g,kh)+V(g,kh)&=&W(hg,k)W(g,h)+V(hg,k)V(g,h)\\
    &=&W(hg,k)W(g,h)+W(hg,k)V(g,h)\\
    &&+V(hg,k)W(g,h)+V(hg,k)V(g,h)\\
    &=&\left(W(hg,k)+V(hg,k)\right)\left(W(g,h)+V(g,h)\right).
    \end{eqnarray*}

   Once again, continuity follows trivially. 
\end{proof}

\begin{lemma}\label{367}
    If $W,V:G\rimes G\to \mathbb{B}_X$ are mutually orthogonal partial cotranslations, then the unit space projector of $W$ is invariant for $W+V$. Moreover, the partial cotranslation obtained through $W+V$ and the projector of the units space of $W$ (following the construction on Proposition \ref{172}) is $W$.
\end{lemma}
\begin{proof}
   It is enough to note that 
   \begin{align*}
       W(hg,e)\left(W(g,h)+V(g,h)\right)&=W(hg,e)W(g,h)+W(hg,e)V(g,h)\\
       &=W(g,h)W(g,e)\\
       &=W(g,h)W(g,e)+V(g,h)W(g,e)\\
       &= \left(W(g,h)+V(g,h)\right)W(g,e).
   \end{align*}

For the second statement, note that $$W(g,h)=W(g,h)W(g,e)=\left(W(g,h)+V(g,h)\right)W(g,e).$$
\end{proof}

\begin{lemma}
    If $\P$ and $\mathrm{Q}$ are mutually orthogonal invariant projectors for a partial cotranslation $V:G\rimes G\to \mathbb{B}_X$, then the partial cotranslations  $V_\P,V_\mathrm{Q}:G\rimes G\to \mathbb{B}_X$, given by $V_\P(g,h)=V(g,h)\P(g)$ and $V_\mathrm{Q}(g,h)=V(g,h)\mathrm{Q}(g)$ are mutually orthogonal.
\end{lemma}

\begin{proof} We know both $V_\P$ and $V_\mathrm{Q}$ are partial cotranslations from Proposition \ref{172}. Moreover
    $$V_P(hg,k) V_\mathrm{Q}(g,h)=V(hg,k)\P(hg)V(g,h)\mathrm{Q}(g)=V(hg,k)V(g,h)\P(g)\mathrm{Q}(g)=0,$$
    and the other composition follows similarly.
\end{proof}

\begin{definition}
    Given two partial cotranslations $W,V:G\times G\to \mathbb{B}_X$, we say they are \textbf{conjugated} if there is a map $T:G\to \mathbb{A}_X$ such that
$$T(hg)V(g,h)=W(g,h)T(g),\qquad\forall\,g,h\in G.$$
   
    In that case we call $T$ a \textbf{conjugation} between $W$ and $V$. If $T$, $W$ and $V$ are continuous, we say they are \textbf{continuously conjugated}. On the other hand, if 
    \begin{equation*}
        \sup_{g\in G}\left\{\norm{T(g)}, \norm{{T(g)^{-1}}}\right\}<\infty,
    \end{equation*}
   then we say $W$ and $V$ are \textbf{boudedly conjugated}
\end{definition}

The interest for conjugate two partial cotranslations is to obtain information of one of them from the other. Moreover, through the end of this section we use this to show that every partial cotranslation is indeed a cotranslation multiplied by some invariant projector.

\smallskip
The interest of finding continuous conjugations is that they would preserve topological properties, while bounded conjugations would preserve asymptotic behavior. 

\begin{lemma}\label{369}
    Let $W:G\times G\to \mathbb{B}_X$ be a partial cotranslation and let $T:G\to \mathbb{A}_X$ be an un arbitrary map. Defining $W_T:G\rimes G\to \mathbb{B}_X$ by
    $$W_T(g,h)=T(hg)^{-1}W(g,h)T(g),$$
    we obtain that $W_T$ is a partial cotranslation. Moreover, if $X$ is finite dimensional, then $\rank W_T=\rank W$. Finally, if $T$ and $W$ are continuous, then $W_T$ is continuous as well.
\end{lemma}
\begin{proof}
    Note that
    \begin{align*}
        W_T(g,kh)&=T(khg)^{-1}W(g,kh)T(g)\\
        &=T(khg)^{-1}W(hg,k)W(g,h)T(g)\\
        &=T(khg)^{-1}W(hg,k)T(hg)T(hg)^{-1}W(g,h)T(g)\\
        &=W_T(hg,k)W_T(g,h).
    \end{align*}

    The second statement is trivial, since every $T(g)$ is invertible. The third statement is trivial.
\end{proof}

\begin{lemma}\label{371}
    If $W,V:G\times G\to \mathbb{B}_X$ are two mutually orthogonal partial cotranslations and $T:G\to \mathbb{A}_X$ is an arbitrary map, then $W_T$ and $V_T$ (as defined on Lemma \ref{369}) are mutually orthogonal. 
\end{lemma}
\begin{proof}
It is enough to see that
\begin{align*}
    W_T(hg,k)V_T(g,h)&=T(khg)^{-1}W(hg,k)T(hg)T(hg)^{-1}V(g,h)T(g)^{-1}\\
    &=T(khg)^{-1}W(hg,k)V(g,h)T(g)^{-1}\\
    &=T(khg)^{-1} \circ 0\circ T(g)^{-1}\\
    &=0,
\end{align*}
and the other composition follows similarly.
\end{proof}

From now on we fix $X=\mathbb{K}^d$ with the Euclidean norm. Thus, we replace $\mathbb{B}_X$ with $\mathcal{M}_d(\mathbb{K})$, the space of square $d\times d$ matrices with entries on the field $\mathbb{K}$, and $\mathbb{A}_X$ with $GL_d(\mathbb{K})$ the general linear group.

\begin{proposition}\label{370}
 Every partial cotranslation $W:G\times G\to \mathcal{M}_d(\mathbb{K})$ is conjugated to a partial cotranslation $\widehat{W}:G\times G\to \mathcal{M}_d(\mathbb{K})$ whose projector of the units space is constant and orthogonal, i.e.
   $$\widehat{W}(g,e)=\begin{pmatrix}
\Id_{\rank W}&0\\
0&0
\end{pmatrix},\qquad\forall\,g\in G.$$
where $\Id_{\rank W}$ is the identity on $\mathbb{K}^{\rank W}$.
\end{proposition}

\begin{proof}
  For every $g\in G$ we know $W(g,e)$ is an idempotent (Lemma \ref{331}) of rank $\rank W$ (Proposition \ref{366}). Then, it follows that there exists some $T(g)\in GL_d(\mathbb{K})$ such that
    $$T(g)^{-1}W(g,e)T(g)=\begin{pmatrix}
\Id_{\rank W}&0\\
0&0
\end{pmatrix},$$
which defines a map $T:G\to GL_d(\mathbb{K})$. Defining $\widehat{W}:G\times G\to \mathbb{B}_X$ by $\widehat{W}=W_T$ as in Lemma \ref{369}, we obtain a partial cotranslation which is by construction conjugated to $W$ and verifies
$$\widehat{W}(g,e)=T(g)^{-1}W(g,e)T(g)=\begin{pmatrix}
\Id_{\rank W}&0\\
0&0
\end{pmatrix}.$$
\end{proof}

Note that the map $T$ we defined on the preceding proposition is not unique. We know eigenvectors (thus diagonalizations) of a continuous matrix function are not in general continuous or bounded. We dedicate the end of this section to deal with this fact.

\begin{proposition}\label{411}
Fix $\P:G\to \mathcal{M}_d(\mathbb{K})$ a projector of constant rank $n$ for which exists $M\geq 1$ such that 
    $$\sup_{g\in G}\left\{\norm{\P(g)},\norm{\Id-\P(g)}\right\}<M,$$
    then, there exists a map $T:G\to GL_d(\mathbb{K})$ such that
    \begin{equation}\label{410}
        T^{-1}(g)\P(g)T(g)=\begin{pmatrix}
\Id_n&0\\
0&0
\end{pmatrix},\qquad\forall\,g\in G,
    \end{equation}
  and
      $$\sup_{g\in G}\left\{\norm{T(g)},\norm{T(g)^{-1}}\right\}<\infty.$$
\end{proposition}

\begin{proof}
    Set $\xi_1,\dots,\xi_d$ the canonical basis of $\mathbb{K}^d$. For each $g\in G$, there is an orthonormal basis of $\im \,\P(g)$ given by $\{\eta_1^g,\dots,\eta_n^g\}$ and an orthonormal basis of $\ker\P(g)$ given by $\{\eta_{n+1}^g,\dots,\eta_d^g\}$. If we define a linear transformation $T(g):\mathbb{K}^d\to \mathbb{K}^d$ by
    $$T(g)\xi_i=\eta_i^g,$$
    evidently it verifies (\ref{410}).

    \smallskip
    
    Let $\zeta\in\mathbb{K}^d$ with $\norm{\zeta}=1$. There exists unique $\alpha_1^\zeta,\dots,\alpha_d^\zeta\in \mathbb{K}$ with $|\alpha_i^\zeta|\leq 1$ such that
    $$\zeta=\sum_{i=1}^d\alpha_i^\zeta\xi_i,$$
then
\begin{align*}
    \norm{T(g)\zeta}=\norm{T(g)\left(\sum_{i=1}^d\alpha_i^\zeta\xi_i\right)}
    \leq\sum_{i=1}^d|\alpha_i^\zeta|\norm{T(g)\xi_i}
    \leq\sum_{i=1}^d\norm{\eta_i^g}=d,
\end{align*}
hence $\norm{T(g)}\leq d$.

\smallskip
On the other hand, there exists unique $\beta_1^{\zeta,g},\dots,\beta_d^{\zeta,g}\in \mathbb{K}$ such that
$$\zeta=\sum_{i=1}^d\beta_i^{\zeta,g}\eta_i^g,$$
hence
$$\P(g)\zeta=\sum_{i=1}^n\beta_i^{\zeta,g}\eta_i^g,\qquad [\Id-\P(g)]\zeta=\sum_{i=n+1}^d\beta_i^{\zeta,g}\eta_i^g.$$

As $\norm{\P(g)}\leq M$, then $\norm{\P(g)\zeta}\leq M$, and as $\{\eta_1^g,\dots,\eta_n^g\}$ is an orthonormal basis, then $|\beta_i^{\zeta,g}|\leq M$ for every $i=1,\dots,n$. Analogously, as $\norm{\Id-\P(g)}\leq M$ and $\{\eta_{n+1}^g,\dots,\eta_d^g\}$ is an orthonormal basis, we obtain  $|\beta_i^{\zeta,g}|\leq M$ for every $i=n+1,\dots,d$. Thus
\begin{align*}
    \norm{T(g)^{-1}\zeta}=\norm{T(g)^{-1}\left(\sum_{i=1}^d\beta_i^{\zeta,g}\eta_i^g\right)} &\leq\sum_{i=1}^d|\beta_i^{\zeta,g}|\norm{T(g)^{-1}\eta_i^g}\\
    &\leq\sum_{i=1}^dM\cdot\norm{\xi_i}=dM,
\end{align*}
hence $\norm{T(g)^{-1}}\leq dM$.
\end{proof}

The next corollary follows trivially from Propositions \ref{370} and \ref{411}.

\begin{corollary}\label{415}
    If $W:G\times G\to \mathcal{M}_d(\mathbb{K})$ is a partial cotranslation whose projector of the units space and its compliment are uniformly bounded, i.e. there exist $M\geq 1$ such that 
    \begin{equation}\label{414}
        \sup_{g\in G}\left\{\norm{W(g,e)},\norm{\Id-W(g,e)}\right\}<M,
    \end{equation}
then $W$ is boundedly conjugated to a partial cotranslation whose units space projector is constant and orthogonal.
\end{corollary}

In this point we present a conjecture that so far we have not been able to proof nor to contradict.

\begin{conjecture}\label{412}
   Denote the canonical basis of $\mathbb{K}^d$ by $\{\xi_1,\dots,\xi_d\}$. If $\P:G\to \mathbb{B}_X$ is a continuous projector, then for every $g\in G$ there exists an orthonormal basis of $\im \,\P(g)$ given by $\{\eta_1^g,\dots,\eta_n^g\}$ and an orthonormal basis of $\ker\P(g)$ given by $\{\eta_{n+1}^g,\dots,\eta_d^g\}$ such that defining $T:G\to GL_d(\mathbb{K})$ by
    $$T(g)\xi_i=\eta_i^g,$$
    we obtain that $T$ is continuous.
\end{conjecture}

Once again a corollary follows trivially from Proposition \ref{370} and the known fact that inversion of linear transformations is a continuous map. The second statement in the following corollary follows trivially from Corollary \ref{415}.

\begin{corollary}\label{413}
     Suppose Conjecture \ref{412} is true. Then, every continuous partial cotranslation $W:G\times G\to \mathbb{B}_X$ is continuously conjugated to a partial cotranslation whose projector of the units space is constant and orthogonal.

     \smallskip
     Moreover, if the projector of the units space of $W$ and its compliment is uniformly bounded, {\it i.e.} it verifies (\ref{414}), then $W$ is continuously and boundedly conjugated to a partial cotranslation whose projector of the units space is constant and orthogonal.
\end{corollary}

\begin{remark}
    {\rm
    Note that the previous corollary provides, if Conjecture \ref{412} is true, a non commutative, non invertible and non differentiable generalization of S. Siegund's result of reducibility for nonautonomous differential equations \cite[Theorem 3.2]{Siegmund2}, since we already showed (Example \ref{417}) that an evolution operator multiplied by an invariant proyector defines a partial cotranslation. The only difference is that in \cite{Siegmund2}, all invariant projectors come from a dichotomy. 

    \smallskip
   Note that if two partial cotranslations are boundedly and continuously conjugated, we are describing a generalized notion of {\it kinematic similarity} (see for instance \cite[Definition 2.1]{Siegmund2}), independently if the group involved is commutative or has a differential structure.
    Remarkably, for discrete groups, as continuity is trivial, the generalization holds.
}
\end{remark}

To conclude, we state a theorem and consequent corollary to summarize the results of this section.

\begin{theorem}\label{368}
   Every partial cotranslation is completable to a cotranslation, i.e. for every partial cotranslation  $W:G\rimes G\to \mathcal{M}_d(\mathbb{K})$ there exists a partial cotranslation $V:G\times G\to \mathcal{M}_d(\mathbb{K})$ such that $W$ and $V$ are mutually orthogonal and $W+V$ has rank $d$. 

   \smallskip
  Moreover, if Conjecture \ref{412} is true and $W$ is continuous, then $V$ can also be chosen continuous.
\end{theorem}
\begin{proof}
Define the constant partial cotranslation $\widehat{V}:G\times G\to \mathcal{M}_d(\mathbb{K})$ by
$$\widehat{V}(g,h)=\begin{pmatrix}
0&0\\
0&\Id_{d-\rank W}
\end{pmatrix}.$$

We affirm $\widehat{V}$ is mutually orthogonal to $\widehat{W}$ (defined as in Proposition \ref{370}). Indeed:
\begin{align*}
\widehat{W}(hg,k)\widehat{V}(g,h)&=\widehat{W}(hg,k)\widehat{W}(hg,e)\widehat{V}(g,h)\\
&=\widehat{W}(hg,k)\begin{pmatrix}
\Id_{\rank W}&0\\
0&0
\end{pmatrix}\begin{pmatrix}
0&0\\
0&\Id_{d-\rank W}
\end{pmatrix}=0,
\end{align*}
and the other composition follows analogously. Choose $T$ as in Proposition \ref{370}. Denoting $T^{\inv}:G\to\mathbb{A}_X$ by $T^\inv(g)=T(g)^{-1}$, it follows by Lemma \ref{371} that $W=\widehat{W}_{T^\inv}$ and $V:=\widehat{V}_{T^{\inv}}$ are mutually orthogonal. Hence, by Lemma \ref{372}, $W+V$ is also a partial cotranslation. Moreover, for an arbitrary $g\in G$:
$$\rank \left(\widehat{W}+\widehat{V}\right)=\rank \left(\widehat{W}(g,e)+\widehat{V}(g,e)\right)=\rank\left(\begin{pmatrix}
\Id_{\rank W}&0\\
0&0
\end{pmatrix}+\begin{pmatrix}
0&0\\
0&\Id_{d-\rank W}
\end{pmatrix}\right)=d,$$
thus, by Lemma \ref{369} we have
$$\rank(W+V)=\rank \left(\widehat{W}+\widehat{V}\right)_{T^\inv}=\rank\left(\widehat{W}+\widehat{V}\right)=d.$$

The second statement follows trivially from Corollary \ref{413}.
\end{proof}

\begin{corollary}
    Every partial cotranslation is a cotranslation multiplied with an invariant projector.
\end{corollary}
\begin{proof}
The statement follows trivially from Theorem \ref{368} and Lemma \ref{367}.
\end{proof}

\end{document}